\documentclass{amsart}%
\usepackage{amsmath}
\usepackage{amssymb}
\usepackage{amsfonts}%
\usepackage{graphicx}
\providecommand{\U}[1]{\protect \rule{.1in}{.1in}}
\newtheorem{theorem}{Theorem}
\theoremstyle{plain}

\newtheorem{definition}{Definition}

\newtheorem{lemma}{Lemma}

\newtheorem{remark}{Remark}

\numberwithin{equation}{section}
\begin{document}
\title[Weighted anisotropic Morrey Spaces estimates]{Weighted anisotropic Morrey Spaces estimates for anisotropic maximal operators }
\author{F. GURBUZ}
\address{ANKARA UNIVERSITY, FACULTY OF SCIENCE, DEPARTMENT OF MATHEMATICS, TANDO\u{G}AN
06100, ANKARA, TURKEY }
\curraddr{}
\email{feritgurbuz84@hotmail.com}
\urladdr{}
\thanks{}
\thanks{}
\thanks{}
\date{}
\dedicatory{ }
\begin{abstract}
The aim of this paper can give weighted anisotropic Morrey Spaces estimates
for anisotropic maximal functions.

\end{abstract}
\maketitle

\section{Introduction}

The classical Morrey spaces have been introduced by Morrey in \cite{Morrey} to
study the local behavior of solutions of second order elliptic partial
differential equations(PDEs). In recent years there has been an explosion of
interest in the study of the boundedness of operators on Morrey-type spaces.
It has been obtained that many properties of solutions to PDEs are concerned
with the boundedness of some operators on Morrey-type spaces. In fact, better
inclusion between Morrey and H\"{o}lder spaces allows to obtain higher
regularity of the solutions to different elliptic and parabolic boundary
problems. Chiarenza and Frasca \cite{Chiarenza} have obtained the boundedness
of Hardy-Littlewood maximal operator, the fractional integral operator and a
singular integral operator in the Morrey spaces. The boundedness of fractional
integral operator has been originally studied by Adams \cite{Adams}. On the
other hand, it is very important to study weighted estimates for these
operators in harmonic analysis. On the weighted $L_{p}$ spaces, the
boundedness of operators above has been obtained by Muckenhoupt
\cite{Muckenhoupt1} and Coifman and Fefferman \cite{Fefferman1}. These results
are extended to several spaces, however, weighted Morrey spaces have yet to be studied.

Thus, in this paper we shall introduce the weighted anisotropic Morrey Spaces
and investigate the boundedness of the anisotropic maximal functions on this space.

\section{Definitions and notation}

Throughout this paper all notation is standard or will be defined as needed.

Let ${\mathbb{R}^{n}}$ be the $n-$dimensional Euclidean space of points
$x=(x_{1},...,x_{n})$ with norm $|x|=\left(
{\displaystyle \sum \limits_{i=1}^{n}}
x_{i}^{2}\right)  ^{\frac{1}{2}}$, $Q=Q\left(  x_{0},r\right)  $ denotes the
cube centered at $x_{0}$ with side length $r$. Given a cube $Q$ and
$\lambda>0$, $\lambda Q$ denotes the cube with the same center as $Q$ whose
side length is $\lambda$ times that of $Q$. A weight is a locally integrable
function on ${\mathbb{R}^{n}}$ which takes values in $(0,\infty)$ almost
everywhere. For a weight function $w$ and a measurable set $E$, we define
$w(E)=%
{\displaystyle \int \limits_{E}}
w(x)dx$, the Lebesgue measure of $E$ by $|E|$ and the characteristic function
of $E$ by $\chi_{_{E}}$. Given a weight function $w$, we say that $w$
satisfies the doubling condition if there exists a constant $D>0$ such that
for any cube $Q $, we have $w(2Q)\leq Dw(Q) $. When $w$ satisfies this
condition, we denote $w\in \Delta_{2}$, for short.

If $w$ is a weight function, we denote by $L_{p}(w)\equiv L_{p}({{\mathbb{R}%
^{n}}},w)$ the weighted Lebesgue space defined by the norm
\[
\Vert f\Vert_{L_{p,w}}=\left(
{\displaystyle \int \limits_{{{\mathbb{R}^{n}}}}}
|f(x)|^{p}w(x)dx\right)  ^{\frac{1}{p}}<\infty,~~~~\text{when }1\leq p<\infty
\]
and by
\[
\Vert f\Vert_{L_{\infty,w}}=\operatorname*{esssup}\limits_{x\in{\mathbb{R}%
^{n}}}|f(x)|w(x),~~~~\text{when }p=\infty.
\]

We denote by $WL_{p}(w)$ the weighted weak space consisting of all measurable
functions $f$ such that%
\[
\Vert f\Vert_{WL_{p}(w)}=\sup \limits_{t>0}tw\left(  \left \{  x\in
{{\mathbb{R}^{n}:}}\left \vert f\left(  x\right)  \right \vert >t\right \}
\right)  ^{\frac{1}{p}}<\infty.
\]

Let $\mathbb{R}_{0}^{n}=\mathbb{R}^{n}\setminus \{0\}$ and $\mathbb{Z}$ be the
set of integer numbers. Let also $a=\left(  a_{1},...,a_{n}\right)  $ be a
fixed vector from $\mathbb{R}^{n}$ with $a_{i}>0$, $i=1,\ldots,n$. Consider a
real $n\times n$ matrix $A$ with eigenvalues $\lambda_{j}$, Re$\lambda
_{j}=a_{j}>0$ and let $Q=trA$ be its trace. The matrix $A$ determines a
one-parameter group $A_{t}=\exp(A\ln t)$, $t>0$ of nonsingular transformations
of $\mathbb{R}^{n}$. Denote by diag $\{a_{1},...,a_{n}\}$ the matrix with
numbers $a_{1},...,a_{n}$ on the main diagonal and zero off -diagonal elements
and let $a_{\max}=\max \limits_{1\leq i\leq n}a_{i}$. Associated with the group
$A_{t}$ is the $A_{t}$-homogeneous metric $\rho:\mathbb{R}_{0}^{n}\rightarrow%
\mathbb{R}
_{+}$, $\rho(A_{t}x)=t\rho(x)$ which is smooth on $\mathbb{R}_{0}^{n}$.

For $x\in \mathbb{R}_{0}^{n}$, let $\left[  x\right]  _{a}$ be a positive
solution to the equation $%
{\displaystyle \sum \limits_{i=1}^{n}}
x_{i}^{2}\left[  x\right]  _{a}^{-2a_{i}}=1$ and $\left \vert x\right \vert
_{a}=\max \limits_{1\leq i\leq n}\left \vert x_{i}\right \vert ^{\frac{1}{a_{i}}%
}$. Note that $\rho \left(  x\right)  $ is equivalent to $\left \vert
x\right \vert _{a}$, i.e.,%
\[
c_{1}\left \vert x\right \vert _{a}\leq \rho \left(  x\right)  \leq c_{2}%
\left \vert x\right \vert _{a}.
\]

For $x\in{\mathbb{R}^{n}}$ and $r>0$ we define the one-parametric
parallelepiped
\begin{align*}
E\left(  x,t\right)   & =\left \{  y\in{\mathbb{R}^{n}:}\left \vert
x-y\right \vert _{a}\leq t\right \} \\
& =\left \{  y\in{\mathbb{R}^{n}:}\left \vert y_{i}-x_{i}\right \vert \leq
t^{a_{i}},\text{ }i=1,\ldots,n\right \}
\end{align*}
and by $E=E\left(  \alpha \right)  $ we denote the set of all $E\left(
x,t\right)  $ with $x\in{\mathbb{R}^{n}}$, $t>0$. If $a_{1}=\cdots=a_{n}$,
then $E\left(  x,t\right)  $ is a cube.

All parallelepipeds are assumed to have their sides parallel to the coordinate
axes. $E=E\left(  x_{0},r\right)  $ denotes the parallelepiped centered at
$x_{0}$ with side length $r^{\alpha_{1}},\ldots,r^{\alpha_{n}} $ consequently.
Given a parallelepiped $E$ and $\lambda>0$, $\lambda^{a}E$ denotes the
parallelepiped with the same center as $E$ whose side length is $\left(
\lambda r\right)  ^{a_{1}},\ldots,\left(  \lambda r\right)  ^{a_{n}}$ consequently.

The letter $C$ is used for various constants, and may change from one
occurrence to another. First we introduce a weighted anisotropic Morrey space.

\begin{definition}
$\left(  \text{\textbf{Weighted anisotropic Morrey spaces}}\right)  $ Let
$1\leq p<\infty$, $0\leq \kappa<1$ and $w$ be a weight. Then a weighted
anisotropic Morrey space is defined by%
\[
L_{p,\kappa,a}(w):=\left \{  f\in L^{loc}(w):\left \Vert f\right \Vert
_{L_{p,\kappa,a}(w)}<\infty \right \}  ,
\]
where%
\[
\left \Vert f\right \Vert _{L_{p,\kappa,a}(w)}=\sup_{E}\left(  \frac{1}{w\left(
E\right)  ^{\kappa}}%
{\displaystyle \int \limits_{E}}
|f(x)|^{p}w(x)dx\right)  ^{\frac{1}{p}}%
\]
and the supremum is taken over all parallelepipeds $E$ on ${\mathbb{R}^{n}}$.
In the case of $a=\left(  1,\ldots,1\right)  $, we get weighted Morrey spaces
$L_{p,\kappa}(w)=L_{p,\kappa,1}(w)$.
\end{definition}

\begin{remark}
Alternatively, we could define the weighted Morrey spaces with anisotropic
balls instead of parallelepipeds. Hence we shall use these two definitions of
weighted anisotropic Morrey spaces appropriate to calculation. Also, we could
define the weighted Morrey spaces with cubes instead of parallelepipeds.
\end{remark}

\begin{remark}
$(1)~$ If $w\equiv1$ and $\kappa=\lambda/n$ with $0<\lambda<n$, then
$L_{p,\kappa}(w)=L_{p,\lambda}({\mathbb{R}^{n}})$ is the classical Morrey
spaces. If $w\equiv1$ and $\kappa=\frac{\lambda}{\left \vert a\right \vert } $
with $0<\lambda<\left \vert a\right \vert $, $\left \vert a\right \vert
=a_{1}+\cdots+a_{n}$, then $L_{p,\kappa,a}(w)=L_{p,\lambda,a}(w)$ the
anisotropic Morrey spaces.

$(2)~$ Let $w\in \Delta_{2}$. If $\kappa=0,$ then $L_{p,0}(w)=L_{p}(w)$ is the
weighted Lebesgue spaces. If $\kappa=1$, then $L_{p,1}(w)=L_{\infty}(w) $ by
the Lebesgue differentiation theorem with respect to $w$ (see \cite{Rudin}).

$(3)~$In the one-dimensional case, let a weight $w\left(  x\right)
=\left \vert x\right \vert ^{\alpha}$ for some $-\frac{1}{2}<\alpha<0$ and a
function $f\left(  x\right)  =\chi_{\left(  0,1\right)  }\left \vert
x\right \vert ^{-\frac{1}{2}}$. Then $f\in L^{1,\frac{\alpha+\frac{1}{2}%
}{\alpha+1}}\left(  w\right)  \setminus L^{2\left(  \alpha+1\right)  }\left(
w\right)  $.
\end{remark}

Let $f\in L^{loc}\left(  {\mathbb{R}^{n}}\right)  $. The anisotropic maximal
function $Mf$ and the sharp maximal function $f^{\#}$ are defined by
\[
Mf(x)=\sup_{t>0}|E(x,t)|^{-1}\int \limits_{E(x,t)}|f(y)|dy
\]
and
\[
f^{\#}\left(  x\right)  =\sup_{t>0}|E(x,t)|^{-1}\int \limits_{E(x,t)}%
|f(y)-f_{E(x,t)}|dy,
\]
where $f_{E(x,t)}=|E(x,t)|^{-1}\int \limits_{E(x,t)}|f(y)|dy$.

If $a_{1}=\cdots=a_{n}$, then $E\left(  x,t\right)  $ is a cube and $Mf$
becomes the usual Hardy-Littlewood maximal function. For $r>0$, we denote
$M_{r}f(x)$ by $\left(  M\left \vert f\right \vert ^{r}\left(  x\right)
\right)  ^{\frac{1}{r}}$.

Let $w$ be a weight. $M_{w}$ denotes the anisotropic maximal operator with
respect to the measure $w\left(  x\right)  dx$ defined by
\begin{equation}
M_{w}f(x)=\sup_{E}\frac{1}{w\left(  E\right)  }%
{\displaystyle \int \limits_{E}}
|f(y)|w\left(  y\right)  dy\label{2*}%
\end{equation}

We shall end this section by defining two weight classes.

\begin{definition}
$\left(  \text{\textbf{Muckenhoupt classes}}\right)  $ A weight function $w$
is in the Muckenhoupt's class $A_{p}\left(  {{\mathbb{R}^{n}}}\right)  $ with
$1<p<\infty$, if there exist $C>1$ such that for any parallelepiped $E$
\begin{equation}
\lbrack w]_{A_{p}\left(  E\right)  }\equiv \left(  |E|^{-1}%
{\displaystyle \int \limits_{E}}
w(x)dx\right)  \left(  |E|^{-1}%
{\displaystyle \int \limits_{E}}
w(x)^{1-p^{\prime}}dx\right)  ^{p-1}\leq C\label{2}%
\end{equation}
where $\frac{1}{p}+\frac{1}{p^{\prime}}=1$ and the infimum of $C$ satisfying
the following inequality (\ref{2}) is denoted by $[w]_{A_{p}}$, and also for
$p=\infty$ we define $A_{\infty}=%
{\displaystyle \bigcup \limits_{1\leq p<\infty}}
A_{p}$, $[w]_{A_{\infty}}=\inf \limits_{1\leq p<\infty}[w]_{A_{p}} $ and
$[w]_{A_{\infty}}\leq \lbrack w]_{A_{p}}$.

When $p=1$, $w\in A_{1}$ if there exist $C>1$ such that for almost every $x$,%
\begin{equation}
Mw(x)\leq Cw(x)\label{1}%
\end{equation}
and the infimum of $C$ satisfying the inequality (\ref{1}) is denoted by
$[w]_{A_{1}}$.
\end{definition}

It is easy to verify that, $\rho \left(  x\right)  ^{\alpha}\in A_{p}$ if and
only if $-\left \vert a\right \vert <\alpha<\left \vert a\right \vert \left(
p-1\right)  $ for $1<p<\infty$ and $\rho \left(  x\right)  ^{\alpha}\in A_{1}$
if and only if $-\left \vert a\right \vert <\alpha \leq0$.

\section{Lemmas and well known results}

In this section, we shall prove some lemmas and describe the well-known result
about the weighted $L_{p}$ spaces.

\begin{theorem}
\label{Theorem1}$\left(  \left[  \text{\cite{HMW}, Theorem 2.6, p.
146}\right]  \right)  $ If $1<p<\infty$ and $w\in \Delta_{2}$, then the
operator $M_{w}$ is bounded on $L_{p}(w)$.
\end{theorem}

The next lemma plays an important role in our proofs of theorems. We say that
$w$ satisfies the reverse doubling condition if $w$ has the property (\ref{3})
of the following lemma.

\begin{lemma}
\label{Lemma3}If $w\in \Delta_{2}$, then there exists a constant $D_{1}>1$ such
that%
\begin{equation}
w\left(  2E\right)  \geq D_{1}w\left(  E\right)  .\label{3}%
\end{equation}

\end{lemma}

\begin{proof}
Fix a parallelepiped $E=E\left(  x_{0},r\right)  $. Then we can choose a
parallelepiped $R\subset2E$ with side length $\frac{r}{2}$ which is disjoint
from the parallelepiped $E$. Hence%
\[
w\left(  E\right)  +w\left(  R\right)  \leq w\left(  2E\right)  .
\]
On the other hand, since $E\subset5R$ we have $w\left(  E\right)  \leq
w\left(  5R\right)  \leq D^{3}w\left(  R\right)  $, where $D$ is a doubling
constant. Therefore we have%
\[
w\left(  E\right)  +\frac{w\left(  E\right)  }{D^{3}}\leq w\left(  2E\right)
.
\]

\end{proof}

\begin{lemma}
\label{Lemma2}(\cite{Kokilashvili}) The following statements hold:

$(1)~$ If $w\in A_{p}$ for some $1\leq p<\infty$, then $w\in \Delta_{2}$.
Moreover, for all $\lambda>1$ we have
\[
w(\lambda E)\leq \lambda^{np}[w]_{A_{p}}w(E).
\]

$(2)~$ Let $w\in A_{p}$ for some $1\leq p<\infty$. Then we have
\[
Mf(x)\leq \lbrack w]_{A_{p}}^{\frac{1}{p}}\left(  M_{w}\left(  \left \vert
f\right \vert ^{p}\right)  \left(  x\right)  \right)  ^{\frac{1}{p}}.
\]

\end{lemma}

\begin{proof}
$\left(  1\right)  $ Let $w\in A_{p}$ for some $1\leq p<\infty$ and
$\lambda>1$. Then%
\begin{align*}
\frac{w\left(  \lambda E\right)  }{w\left(  E\right)  }  & =\left(
\frac{\left \vert \lambda E\right \vert }{\left \vert E\right \vert }\right)
^{p}\frac{[w]_{A_{p}}\left(  \lambda E\right)  }{[w]_{A_{p}}\left(  E\right)
}\frac{\left(
{\displaystyle \int \limits_{E}}
w\left(  x\right)  ^{1-p^{\prime}}dx\right)  ^{p-1}}{\left(
{\displaystyle \int \limits_{\lambda E}}
w\left(  x\right)  ^{1-p^{\prime}}dx\right)  ^{p-1}}\\
& \leq \lambda^{np}[w]_{A_{p}}.
\end{align*}

$\left(  2\right)  $ Let $w\in A_{p}$ for some $1\leq p<\infty$. Applying the
H\"{o}lder's inequality, we get
\begin{align*}
Mf(x)  & =\sup \limits_{E}\frac{1}{\left \vert E\right \vert }\int \limits_{E}%
|f(x)|dx\\
& \leq \sup \limits_{E}\frac{1}{\left \vert E\right \vert }\left(
{\displaystyle \int \limits_{E}}
\left \vert f\left(  x\right)  \right \vert ^{p}w(x)dx\right)  ^{\frac{1}{p}%
}\left(
{\displaystyle \int \limits_{E}}
w(x)^{1-p^{\prime}}dx\right)  ^{\frac{1}{p^{\prime}}}\\
& =\sup \limits_{E}\left(  \frac{1}{w\left(  E\right)  }%
{\displaystyle \int \limits_{E}}
\left \vert f\left(  x\right)  \right \vert ^{p}w(x)dx\right)  ^{\frac{1}{p}%
}\left(
{\displaystyle \int \limits_{E}}
w(x)^{1-p^{\prime}}dx\right)  ^{\frac{1}{p^{\prime}}}\left(  \frac{w\left(
E\right)  }{\left \vert E\right \vert }\right)  ^{\frac{1}{p}}\\
& =\left(  \sup \limits_{E}\frac{1}{w\left(  E\right)  }%
{\displaystyle \int \limits_{E}}
\left \vert f\left(  x\right)  \right \vert ^{p}w(x)dx\left(  \frac
{1}{\left \vert E\right \vert }%
{\displaystyle \int \limits_{E}}
w(x)dx\right)  \left(
{\displaystyle \int \limits_{E}}
w(x)^{1-p^{\prime}}dx\right)  ^{p-1}\right)  ^{\frac{1}{p}}\\
& \leq[w]_{A_{p}}^{\frac{1}{p}}\left(  M_{w}\left(  \left \vert f\right \vert
^{p}\right)  \left(  x\right)  \right)  ^{\frac{1}{p}}.
\end{align*}

\end{proof}

\section{Anisotropic maximal function}

In this section, we shall state the boundedness of the anisotropic maximal
operators on weighted anisotropic Morrey Spaces.

\begin{theorem}
\label{Theorem2}(Our main result) If $1<p<\infty$, $0<\kappa<1$ and
$w\in \Delta_{2}$, then the operator $M_{w}$ is bounded on $L_{p,\kappa,a}(w)$.
\end{theorem}

\begin{proof}
Fix a parallelepiped $E\subset{\mathbb{R}^{n}}$. We decompose $f=f_{1}+f_{2}
$, where $f_{1}=f\chi_{3E}$. Since $M_{w}$ is a sublinear operator, we have%
\begin{align*}
& \left(
{\displaystyle \int \limits_{E}}
M_{w}f\left(  x\right)  ^{p}w(x)dx\right)  ^{\frac{1}{p}}\\
& \leq \left(
{\displaystyle \int \limits_{E}}
M_{w}f_{1}\left(  x\right)  ^{p}w(x)dx\right)  ^{\frac{1}{p}}+\left(
{\displaystyle \int \limits_{E}}
M_{w}f_{2}\left(  x\right)  ^{p}w(x)dx\right)  ^{\frac{1}{p}}\\
& =I+II.
\end{align*}

For the term $I$, since $M_{w}$ is bounded on $L_{p}(w)$ (see Theorem
\ref{Theorem1}), we obtain%
\begin{align*}
I  & \leq \left(
{\displaystyle \int \limits_{E}}
M_{w}f_{1}\left(  x\right)  ^{p}w(x)dx\right)  ^{\frac{1}{p}}\\
& \leq C\left(
{\displaystyle \int \limits_{3E}}
|f(x)|^{p}w(x)dx\right)  ^{\frac{1}{p}}\\
& \leq C\left \Vert f\right \Vert _{L_{p,\kappa,a}(w)}w\left(  E\right)
^{\frac{\kappa}{p}}.
\end{align*}

Next we estimate the term $II$. By simple geometric observation, we have for
any $x\in E$, note that for all $\tilde{E}$ such that $x\in \tilde{E}$,
$\tilde{E}\cap \left(  3E\right)  ^{c}\neq \emptyset$ there exists $R$ such that
$E\subset3R$ and $\frac{1}{3}R\subset \tilde{E}\subset R$.

Then $w\left(  \tilde{E}\right)  \geq w\left(  \frac{1}{3}R\right)  \geq
\frac{1}{D_{1}}w\left(  R\right)  $ and%
\[
M_{w}f_{2}(x)\leq D_{1}\sup_{R:E\subset3R}\frac{1}{w\left(  R\right)  }%
{\displaystyle \int \limits_{R}}
|f(y)|w\left(  y\right)  dy.
\]

Note that%
\begin{align*}
& \frac{1}{w\left(  R\right)  }%
{\displaystyle \int \limits_{R}}
|f(y)|w\left(  y\right)  dy\\
& \leq \frac{1}{w\left(  R\right)  }\left(
{\displaystyle \int \limits_{R}}
\left \vert f\left(  y\right)  \right \vert ^{p}w(y)dy\right)  ^{\frac{1}{p}%
}\left(
{\displaystyle \int \limits_{R}}
w(y)dy\right)  ^{\frac{1}{p^{\prime}}}\\
& \leq \frac{1}{w\left(  R\right)  ^{1-\frac{1}{p^{\prime}}}}\left(
{\displaystyle \int \limits_{R}}
\left \vert f\left(  y\right)  \right \vert ^{p}w(y)dy\right)  ^{\frac{1}{p}}\\
& =\frac{1}{w\left(  R\right)  ^{\frac{1}{p}}}\left(  \frac{1}{w\left(
R\right)  ^{\kappa}}%
{\displaystyle \int \limits_{R}}
\left \vert f\left(  y\right)  \right \vert ^{p}w(y)dy\right)  ^{\frac{1}{p}%
}w\left(  R\right)  ^{\frac{\kappa}{p}}\\
& \leq \left(  \frac{1}{w\left(  R\right)  ^{\kappa}}%
{\displaystyle \int \limits_{R}}
\left \vert f\left(  y\right)  \right \vert ^{p}w(y)dy\right)  ^{\frac{1}{p}%
}w\left(  R\right)  ^{\frac{\kappa-1}{p}}\\
& \leq C\left \Vert f\right \Vert _{L_{p,\kappa,a}(w)}w\left(  R\right)
^{\frac{\kappa-1}{p}},
\end{align*}
if $E\subset3R$. So we obtain%
\begin{align*}
II  & =\left(
{\displaystyle \int \limits_{E}}
M_{w}f_{2}\left(  x\right)  ^{p}w(x)dx\right)  ^{\frac{1}{p}}\\
& \leq \left(
{\displaystyle \int \limits_{E}}
\left[  \sup \limits_{x\in R}\frac{1}{w\left(  R\right)  }%
{\displaystyle \int \limits_{R}}
|f(y)|w\left(  y\right)  dy\right]  ^{p}w(x)dx\right)  ^{\frac{1}{p}}\\
& \leq \left(
{\displaystyle \int \limits_{E}}
\left[  c\left \Vert f\right \Vert _{L_{p,\kappa,a}(w)}w\left(  E\right)
^{\frac{\kappa-1}{p}}\right]  ^{p}w(x)dx\right)  ^{\frac{1}{p}}\\
& \leq c\left \Vert f\right \Vert _{L_{p,\kappa,a}(w)}w\left(  E\right)
^{\frac{\kappa-1}{p}}\left(
{\displaystyle \int \limits_{E}}
w(x)dx\right)  ^{\frac{1}{p}}\\
& =c\left \Vert f\right \Vert _{L_{p,\kappa,a}(w)}w\left(  E\right)
^{\frac{\kappa}{p}}.
\end{align*}
Therefore%
\[
II\leq C\left \Vert f\right \Vert _{L_{p,\kappa,a}(w)}w\left(  E\right)
^{\frac{\kappa}{p}}.
\]
This completes the proof.
\end{proof}

\begin{theorem}
\label{Theorem3}(Our main result) If $1<p<\infty$, $0<\kappa<1$ and $w\in
A_{p}$, then the anisotropic maximal operator $M$ is bounded on $L_{p,\kappa
,a}(w)$.

If $p=1$, $0<\kappa<1$ and $w\in A_{1}$, for all $t>0$ and any parallelepiped
$E$,%
\[
w\left(  x\in E:Mf\left(  x\right)  >t\right)  \leq \frac{C}{t}\left \Vert
f\right \Vert _{L_{1,\kappa,a}(w)}w\left(  E\right)  ^{\kappa}.
\]

\end{theorem}

\begin{proof}
Let $1<p<\infty$. By the reverse H\"{o}lder inequality (see \cite{Fefferman1}%
), there exists $1<r<p$ such that $w\in A_{r}$. Hence it follows from Lemma
\ref{Lemma2} $\left(  2\right)  $ and Theorem \ref{Theorem2} that%
\begin{align*}
& \left(  \frac{1}{w\left(  E\right)  ^{\kappa}}%
{\displaystyle \int \limits_{E}}
Mf\left(  x\right)  ^{p}w(x)dx\right)  ^{\frac{1}{p}}\\
& \leq C\left(  \frac{1}{w\left(  E\right)  ^{\kappa}}%
{\displaystyle \int \limits_{E}}
M_{w}\left(  \left \vert f\right \vert ^{r}\right)  \left(  x\right)  ^{\frac
{p}{r}}w(x)dx\right)  ^{\frac{1}{p}}\\
& \leq C\left \Vert M_{w}\left(  \left \vert f\right \vert ^{r}\right)
\right \Vert _{L_{\frac{p}{r},\kappa,a}\left(  w\right)  }^{\frac{1}{r}}\\
& \leq C\left \Vert \left \vert f\right \vert ^{r}\right \Vert _{L_{\frac{p}%
{r},\kappa,a}\left(  w\right)  }^{\frac{1}{r}}\\
& =c\left[  \sup_{E}\left(  \frac{1}{w\left(  E\right)  ^{\kappa}}%
{\displaystyle \int \limits_{E}}
|f(x)|^{r\left(  \frac{p}{r}\right)  }w(x)dx\right)  ^{\frac{r}{p}}\right]
^{\frac{1}{r}}\\
& =C\left \Vert f\right \Vert _{L_{p,\kappa,a}(w)}.
\end{align*}

When $p=1$, we use the Fefferman-Stein maximal inequality%
\[%
{\displaystyle \int \limits_{\left \{  x:Mf\left(  x\right)  >t\right \}  }}
\varphi \left(  x\right)  dx\leq \frac{C}{t}%
{\displaystyle \int \limits_{\mathbb{R}^{n}}}
\left \vert f\left(  x\right)  \right \vert \left(  M_{\varphi}\right)  \left(
x\right)  dx
\]
for any functions $f$ and $\varphi \geq0$ (see \cite{Fefferman3, HMW}).

Fix a parallelepiped $E=E\left(  x_{0},r\right)  $. Put $\varphi \left(
x\right)  =w\left(  x\right)  \chi_{E}\left(  x\right)  $. Then we have
\begin{align*}
&
{\displaystyle \int \limits_{\left \{  x:Mf\left(  x\right)  >t\right \}  }}
\chi_{E}\left(  x\right)  w\left(  x\right)  dx\\
& \leq \frac{C}{t}%
{\displaystyle \int \limits_{\mathbb{R}^{n}}}
\left \vert f\left(  x\right)  \right \vert M\left(  w\chi_{E}\right)  \left(
x\right)  dx\\
& =\frac{C}{t}\left(
{\displaystyle \int \limits_{3E}}
+%
{\displaystyle \int \limits_{\left(  3E\right)  ^{c}}}
\right)  =\frac{C}{t}\left \{  I+II\right \}
\end{align*}
for all $t>0$.

We now estimate the term $I$. Since $w\in A_{1}$, it follows that%
\[
M\left(  w\chi_{E}\right)  \left(  x\right)  \leq M\left(  w\right)  \left(
x\right)  \leq Cw\left(  x\right)  .
\]
So it follows that%
\[
I\leq Cw\left(  3E\right)  ^{\kappa}\left \Vert f\right \Vert _{L_{1,\kappa
,a}(w)}\leq Cw\left(  E\right)  ^{\kappa}\left \Vert f\right \Vert
_{L_{1,\kappa,a}(w)}.
\]

To estimate the term $II$, we consider the form%
\[
\frac{1}{\left \vert R\right \vert }%
{\displaystyle \int \limits_{R\cap E}}
w\left(  y\right)  dy
\]
for $x\in \left(  3E\right)  ^{c}\cap R$ and $R\cap E\neq \emptyset$. By simple
geometric observation, we have%
\[
\frac{1}{\left \vert R\right \vert }%
{\displaystyle \int \limits_{R\cap E}}
w\left(  y\right)  dy\leq C_{n}\left(  \frac{1}{\left \vert x-x_{0}\right \vert
_{a}^{\left \vert a\right \vert }}%
{\displaystyle \int \limits_{E}}
w\left(  y\right)  dy\right)  \leq C_{n}\left \vert x-x_{0}\right \vert
_{a}^{-\left \vert a\right \vert }w\left(  E\right)  .
\]
Therefore we obtain%
\[
M\left(  w\chi_{E}\right)  \left(  x\right)  \leq C_{n}\left \vert
x-x_{0}\right \vert _{a}^{-\left \vert a\right \vert }w\left(  E\right)  .
\]
Since $w\in A_{1}$, we have $w\in \Delta_{2}$ by Lemma \ref{Lemma2} $\left(
1\right)  $. Using Lemma \ref{Lemma3}, we have $w\left(  3E\right)  \geq
w\left(  2E\right)  \geq D_{1}w\left(  E\right)  $ with $D_{1}>1$. Thus we can
estimate the term $II$ as follows:%
\begin{align*}
II  & \leq Cw\left(  E\right)
{\displaystyle \int \limits_{\left(  3E\right)  ^{c}}}
\frac{\left \vert f\left(  x\right)  \right \vert }{\left \vert x-x_{0}%
\right \vert _{a}^{\left \vert a\right \vert }}dx\\
& =C%
{\displaystyle \sum \limits_{j=1}^{\infty}}
{\displaystyle \int \limits_{3^{j+1}E\setminus \left(  3^{j}E\right)  }}
\frac{\left \vert f\left(  x\right)  \right \vert }{\left \vert x-x_{0}%
\right \vert _{a}^{\left \vert a\right \vert }}dx\\
& \leq Cw\left(  E\right)
{\displaystyle \sum \limits_{j=1}^{\infty}}
\frac{1}{\left \vert 3^{j}E\right \vert }%
{\displaystyle \int \limits_{3^{j+1}E}}
\left \vert f\left(  x\right)  \right \vert dx\\
& \leq Cw\left(  E\right)
{\displaystyle \sum \limits_{j=1}^{\infty}}
\frac{1}{\left \vert 3^{j}E\right \vert }\left(  \operatorname*{esssup}%
\limits_{x\in3^{j+1}E}\frac{1}{w\left(  x\right)  }\right)
{\displaystyle \int \limits_{3^{j+1}E}}
\left \vert f\left(  x\right)  \right \vert w\left(  x\right)  dx\\
& =Cw\left(  E\right)
{\displaystyle \sum \limits_{j=1}^{\infty}}
\frac{1}{\left \vert 3^{j}E\right \vert }\frac{\left \vert 3^{j+1}E\right \vert
}{w\left(  3^{j+1}E\right)  }%
{\displaystyle \int \limits_{3^{j+1}E}}
\left \vert f\left(  x\right)  \right \vert w\left(  x\right)  dx\\
& =Cw\left(  E\right)  ^{\kappa}%
{\displaystyle \sum \limits_{j=1}^{\infty}}
\frac{w\left(  E\right)  ^{1-\kappa}}{w\left(  3^{j+1}E\right)  ^{1-\kappa}%
}\frac{1}{w\left(  3^{j+1}E\right)  ^{\kappa}}%
{\displaystyle \int \limits_{3^{j+1}E}}
\left \vert f\left(  x\right)  \right \vert w\left(  x\right)  dx\\
& \leq Cw\left(  E\right)  ^{\kappa}\left \Vert f\right \Vert _{L_{1,\kappa
,a}(w)}%
{\displaystyle \sum \limits_{j=1}^{\infty}}
\frac{w\left(  E\right)  ^{1-\kappa}}{w\left(  3^{j+1}E\right)  ^{1-\kappa}}\\
& \leq Cw\left(  E\right)  ^{\kappa}\left \Vert f\right \Vert _{L_{1,\kappa
,a}(w)}.
\end{align*}

The last series converges since the reverse doubling constant is larger than
one (see Lemma \ref{Lemma3}). This completes the proof.
\end{proof}

\end{document}